\documentclass[11pt]{article}
\usepackage[dvips]{graphicx}
\usepackage{amssymb,color}
\usepackage[centertags]{amsmath}
\usepackage{amsfonts}

\def\BBox{\kern  -0.2cm\hbox{\vrule width 0.2cm height 0.2cm}}

\title{A note on extremal results on directed acyclic graphs}
\author{A. Mart\'{i}nez-P\'{e}rez\footnote{Departamento de An\'alisis Econ\'omico y Finanzas, Universidad de Castilla-La Mancha,
Avda. Real F\'abrica de Seda, s/n. 45600 Talavera de la Reina, Toledo, Spain. alvaro.martinezperez@uclm.es}, L. Montejano\footnote{Instituto de Matem\'aticas, Universidad Nacional Aut\'onoma de M\'exico, \'Area de la Investigaci\'on Cient\'ifica, Circuito Exterior, Cu. Coyoacan 04510, M\'exico D.F., Mexico. luis@matem.unam.mx} and D. Oliveros \footnote{Instituto de Matem\'aticas, Universidad Nacional Aut\'onoma de M\'exico, \'Area de la Investigaci\'on Cient\'ifica, Circuito Exterior, Cu. Coyoacan 04510, M\'exico D.F., Mexico. dolivero@matem.unam.mx}}

\newtheorem{definicion}{Definition}[section]

\newtheorem{prop}[definicion]{Proposition}
\newtheorem{lema}[definicion]{Lemma}
\newtheorem{obs}[definicion]{Remark}
\newtheorem{teorema}[definicion]{Theorem}
\newtheorem{cor}[definicion]{Corollary}

\newcommand{\co}{\ensuremath{\colon}} 
\newcommand{\bn}{\ensuremath{\mathbb{N}}} 
\newcommand{\br}{\ensuremath{\mathbb{R}}} 

\begin{document}


\maketitle

\begin{abstract}
The family of Directed Acyclic Graphs as well as some related graphs are analyzed with respect to extremal behavior in relation with the family of intersection graphs 
for families of boxes with transverse intersection.
\end{abstract}

\section{Introduction}

One of the fundamental results in graph theory which initiated extremal graph theory is the Theorem of Tur\'an (1941) which states that a graph with 
$n$ vertices that has more than $T(n,k)$ edges, will always contain a complete subgraph of size $k+1$.  
The {\it Tur\'an number}, $T(n,k)$ is defined as the maximum number of edges in a graph with $n$ vertices without a clique of size $k+1$.
It is known that $T(n,k)\leq (1-\frac{1}{k})\frac{n^2}{2}$, and
equality holds if $k$ divides $n$. In fact, $\lim_{n\to\infty}\frac{T(n,m)}{\frac{n^2}{2}}=1-\frac{1}{m}. $  See \cite{AZ}.

Tur\'an numbers for several families graphs have been studied in the context of extremal graph  theory, see for example \cite{B} and  \cite{D}.  In (\cite{MMO}, \cite{BFMMOP}) the authors analyze, among other things, 
the intersection graphs of boxes in $\br^d$ proving that, if $\mathcal{T}(n,k,d)$ denotes the maximal number of intersection pairs in a family $\mathcal{F}$ of $n$ 
boxes in $\br^d$ with the property that no $k+1$ boxes in $\mathcal{F}$ have a point in common (with $n\geq k\geq d\geq 1)$, then  
$\mathcal{T}(n,k,d)=\mathcal{T}(n-k+d,d)+\mathcal{T}(n,k-d+1,1)$, being $\mathcal{T}(n,k,1)={n\choose 2}-{n-k+1\choose 2}$ the precise bound in dimension 
$1$ for the family of interval graphs. 

Tur\'an numbers have played and important role for several variants of the Tur\'an Theorem and its relation with the fractional Helly Theorem (see \cite{K}, \cite{KL1979}).

The purpose of this paper is to study some extremal results and their connection with the Tur\'an numbers for the family of  directed acyclic graphs.
This is related with the extremal behavior of the family of intersection graphs for a collection of boxes in $br^2$ with transverse intersection.

The first result, Proposition \ref{Prop: Turan}, states that in a directed acyclic graph with $n$ vertices, 
if the longest path has length $\ell$, then the maximal number of edges is the Tur\'an number $T(n,\ell+1)$.

Theorem \ref{Theorem: bound} and its corollaries state that given a  Directed Aciclic Graph  $\vec{G}$ with $n$ vertices such that the longest path has 
length $\ell$ then, if $\vec{G}$ is either reduced, strongly reduced or extremely reduced, $\vec{G}$ has at most $\mathcal{T}(n-\ell+1,2)+T(n,\ell,1)$ edges, where again $\mathcal{T}(n,\ell,1)$ denotes the maximal number of intersecting pairs in a family $\mathcal{F}$ of $n$ intervals in 
$\mathbb{R}$ with the property that no $\ell+1$ intervals in $\mathcal{F}$ have a point in common. 

In fact, this bound is best possible.
The bound is reached by the intersection graph of a collection of boxes in $\br^2$ with transverse intersection. This graph is reduced, strongly reduced and extremely reduced.

\section{Directed acyclic graphs}

By a \textit{directed acyclic graph}, DAG, we mean a simple directed graph without directed cycles. A DAG, $\vec{G}=(\mathcal{V},\vec{\mathcal{E}})$, with vertex set $\mathcal{V}$ and directed edge set $\vec{\mathcal{E}}$ is \textit{transitive} if for every $x,y,z\in \mathcal{V}$, if $\{x,y\},\{y,z\}\in \vec{\mathcal{E}}$ then $\{x,z\}\in \vec{\mathcal{E}}$. 

\begin{definicion} A \textit{topological order} of a directed graph $\vec{G}$ is an ordering of its vertices $\{v_1,v_2, ... ,v_n\}$ so that for every edge $\{v_i, v_j\}$ then $i < j$.
\end{definicion}

The following proposition is a well known result:

\begin{prop} A directed graph $\vec{G}$ is a DAG if and only if $\vec{G}$ has a topological order.
\end{prop}

Given any set $X$, by $|X|$ we denote the cardinal of $X$.

The \textit{indegree}, $deg^-(v)$, of a vertex $v$ is the number of directed edges $\{x,v\}$ with $x \in V$. The \textit{outdegree}, $deg^+(v)$, of a vertex $v$ is the number of directed edges $\{v,x\}$ with $x \in \mathcal{V}$. Notice that each direct edge $\{v,w\}$ adds one outdegree to the vertex $v$ and one indegree to the vertex $w$. Therefore, $\sum_{v\in \mathcal{V}} deg^+(v)=\sum_{v\in \mathcal{V}} deg^-(v)=|(\vec{\mathcal{E}})|$.

The degree of a vertex is $deg(v)=deg^-(v)+deg^+(v)$.

A vertex $v$ such that  $deg^-(v)=0$ is called \textit{source}. A vertex $v$ such that  $deg^+(v)=0$ is called \textit{sink}. It is well known, that every DAG $\vec{G}$ has at least one source and one sink.

Given a DAG, $\vec{G}=(\mathcal{V},\vec{\mathcal{E}})$, a directed path $\vec{\gamma}$ in $G$ is a sequence of vertices $\{v_0,...,v_n\}$ such that $\{v_{i-1},v_i\}\in \vec{\mathcal{E}}$ for every $1\leq i\leq n$. Here, $\vec{\gamma}$ has \emph{length} $n$, and \emph{endpoint} $v_n$.

Given a DAG, $\vec{G}=(\mathcal{V},\vec{\mathcal{E}})$, let $\Gamma\co \mathcal{V}\to \bn$ be such that $\Gamma(v)=k$ if there exists a directed path $\vec{\gamma}$ in $G$ of length $k$ with endpoint $v$ and there is no directed path $\vec{\gamma'}$ with endpoint $v$ and length greater than $k$.

Given a DAG, $\vec{G}=(\mathcal{V},\vec{\mathcal{E}})$ suppose that $\ell=\max\{k \, | \, \Gamma(v)=k\ {\text{for every }} v\in \mathcal{V}\}$. Notice that, since $\vec{G}$ has no directed cycle, $\ell \leq |\mathcal{V}|$. Then, let us define a partition $P_\Gamma=\{V_0,...,V_\ell\}$ of $\mathcal{V}$ such that $V_i:=\{v\in \mathcal{V} \, | \, \Gamma(v)=i\}$ for every $0\leq i \leq \ell$.

Notice that $V_0$ is exactly the set of sources in $\vec{G}$ and $V_\ell$ is contained in the set of sinks in $G$.

\begin{prop} $V_i$ is nonempty for every $0\leq i \leq \ell$.
\end{prop}

\begin{proof} Let $\{v_0,...,v_\ell\}$ be a directed path of maximal length in $\vec{G}$. Clearly, for every $0\leq i \leq \ell$, $v_i\notin V_j$ if $j<i$. Suppose $v_i\in V_j$ with $i<j\leq \ell$. Then, there is a directed path $\{v'_0,...,v'_j=v_i\}$ with $j>i$ and $\{v'_0,...,v'_j,v_{i+1},...,v_\ell\}$ is a directed path with length $j+l-i>\ell$ which contradicts the hypothesis.
\end{proof}

\begin{prop}\label{Prop: indep} The induced subgraph with vertices $V_i$, $G[V_i]$, is independent (has no edges) for every $i$.
\end{prop}

\begin{proof} Let $v_i,v'_i\in V_i$ and suppose $\{v_i,v'_i\} \in \vec{\mathcal{E}}$. Let $\{v_0,...,v_i\}$ be a path of length $i$ with endpoint $v_i$. Then, $\{v_0,...,v_i,v'_i\}$ defines a directed path of length $i+1$ which contradicts the fact that $v'_i\in V_{i}$.
\end{proof}

Let $T(n,\ell)$ denote the $\ell$-partite Tur\'an graph with $n$ vertices and let $t(n,\ell)$ denote the number of edges of $T(n,\ell)$.

\begin{prop}\label{Prop: Turan}  Let $\vec{G}=(\mathcal{V},\vec{\mathcal{E}})$ be a DAG with $n$ vertices and such that the longest directed path has length $\ell$. Then, $\vec{G}$ has at most $t(n,\ell+1)$ edges. 
\end{prop}

\begin{proof} Consider the partition $P_\Gamma=\{V_0,...,V_\ell\}$ of $\mathcal{V}$. By Proposition \ref{Prop: indep}, this defines a 
$(\ell+1)$-partite directed graph. Thus, neglecting the orientation we obtain a complete $(\ell+1)$-partite graph with partition sets 
$V_0,...,V_\ell$. Therefore, the number of edges is at most $t(n,\ell+1)$.
\end{proof}

\begin{obs} It is readily seen that the bound in Proposition \ref{Prop: Turan} is best possible. Consider the Tur\'an graph $T(n,\ell+1)$ and any ordering of the $\ell+1$ independent sets $V_0,...,V_\ell$. Then, for every edge  $\{v_i,v_j\}$ in $T(n,\ell)$ with $v_i\in V_i$, 
$v_j\in V_j$ and $i<j$ let us assume the orientation $\{v_i,v_j\}$. It is trivial to check that the resulting graph is a DAG with 
$t(n,\ell+1)$ edges. 
\end{obs}

\section{Reduced, strongly reduced and extremely reduced DAG.}

Let $\mathcal{O}$ be a topological ordering in a DAG $\vec{G}$. Given any two vertices $v,w$, and two directed paths in $\vec{G}$, $\gamma$,$\gamma'$, from $v$ to $w$, let us define $\gamma\cup_\mathcal{O} \gamma'$ as the sequence of vertices defined by the vertices in $\gamma\cup \gamma'$ in the order given by $\mathcal{O}$. Of course, this need not be, in general, a directed path from $v$ to $w$.  

Let $\Gamma(u,v)$ be the set of all directed paths from $u$ to $v$.  Let  $\cup_{\mathcal{O}}\{\gamma \, | \, \gamma \in \Gamma(u,v)\}$ represent the sequence of all the vertices from the paths in $\Gamma(u,v)$ ordered according to $\mathcal{O}$.

\begin{definicion} A finite DAG $\vec{G}$ is \emph{strongly reduced} if for any topological ordering $\mathcal{O}$ of $\vec{G}$, every pair of vertices, $v,w$, and every pair of directed paths,	$\gamma,\gamma'$, from $v$ to $w$, 	then $\gamma\cup_\mathcal{O} \gamma'$ defines a directed path from $v$ to $w$.
\end{definicion}

Let $\vec{G}$ be DAG. Given any two vertices $v,w$, and two directed paths in $\vec{G}$, $\gamma$,$\gamma'$, from $v$ to $w$, let us define 
$\gamma \leq \gamma'$ if every vertex in $\gamma$ is also in $\gamma'$. Clearly, $``\leq "$ is a partial order.

A vertex $w$ is \emph{reachable} from a vertex $v$ if there is a directed path from $v$ to $w$.

\begin{prop}\label{Prop: reduced} Given a finite DAG $\vec{G}=(\mathcal{V},\vec{\mathcal{E}})$, the following properties are equivalent:
\begin{itemize}
	\item[i)] For every pair of vertices $v,w$ and every pair of paths,	$\gamma,\gamma'$, from $v$ to $w$, there exists a directed path from $v$ to $w$, $\gamma''$, such that $\gamma, \gamma'\leq \gamma''$. 
	\item[ii)] For every pair of vertices $v,w$ such that $w$ is reachable from $v$, there is a directed path from $v$ to $w$, 
	$\gamma_M$, such that for every directed path, $\gamma$, from $v$ to $w$, $\gamma \leq \gamma_M$.
	\item[iii)] For every topological ordering $\mathcal{O}$ of $\vec{G}$ and any pair of vertices $v,w$,  $\cup_{\mathcal{O}}\{\gamma \, | \, \gamma \in \Gamma(u,v)\}$ defines a directed path from $v$ to $w$.
\end{itemize}
\end{prop}

\begin{proof} Since the graph is finite and the relation '$\leq$' is transitive, $i)$ and $ii)$ are trivially equivalent.

If $ii)$ is satisfied, then it is trivial to see that  $\cup_{\mathcal{O}}\{\gamma \, | \, \gamma \in \Gamma(u,v)\}=\gamma_M$ and $iii)$ is satisfied. Also, it is readily seen that $iii)$ implies $ii)$ taking $\gamma_M:=\cup_{\mathcal{O}}\{\gamma \, | \, \gamma \in \Gamma(u,v)\}$.
\end{proof}

\begin{definicion} We say that a finite DAG $\vec{G}$ is \emph{reduced} if it satisfies any of the properties from Proposition 
\ref{Prop: reduced}.
\end{definicion}

\begin{prop}\label{Prop: strongly-reduced} If a finite DAG $\vec{G}$ is strongly reduced, then $\vec{G}$ is reduced.
\end{prop}

\begin{proof} Since the graph is finite, it is immediate to see that being strongly reduced implies $iii)$.
\end{proof}

\begin{obs} The converse is not true. The graph in the left from Figure \ref{Ejp_1} is clearly reduced. Notice that the directed path $\gamma_M:=\{v_1,v_2,v_3,v_4,v_5\}$ is an upper bound for every directed path from $v_1$ to $v_5$. However, if we consider the directed paths $\gamma=\{v_1,v_2,v_5\}$ and $\gamma'=\{v_1,v_4,v_5\}$ with the topological order $\mathcal{O}=\{v_1,v_2,v_3,v_4,v_5\}$, then $\gamma \cup_\mathcal{O}\gamma'=\{v_1,v_2,v_4,v_5\}$ which is not a directed path.
\end{obs}

\begin{figure}[ht]
\centering
\includegraphics[scale=0.4]{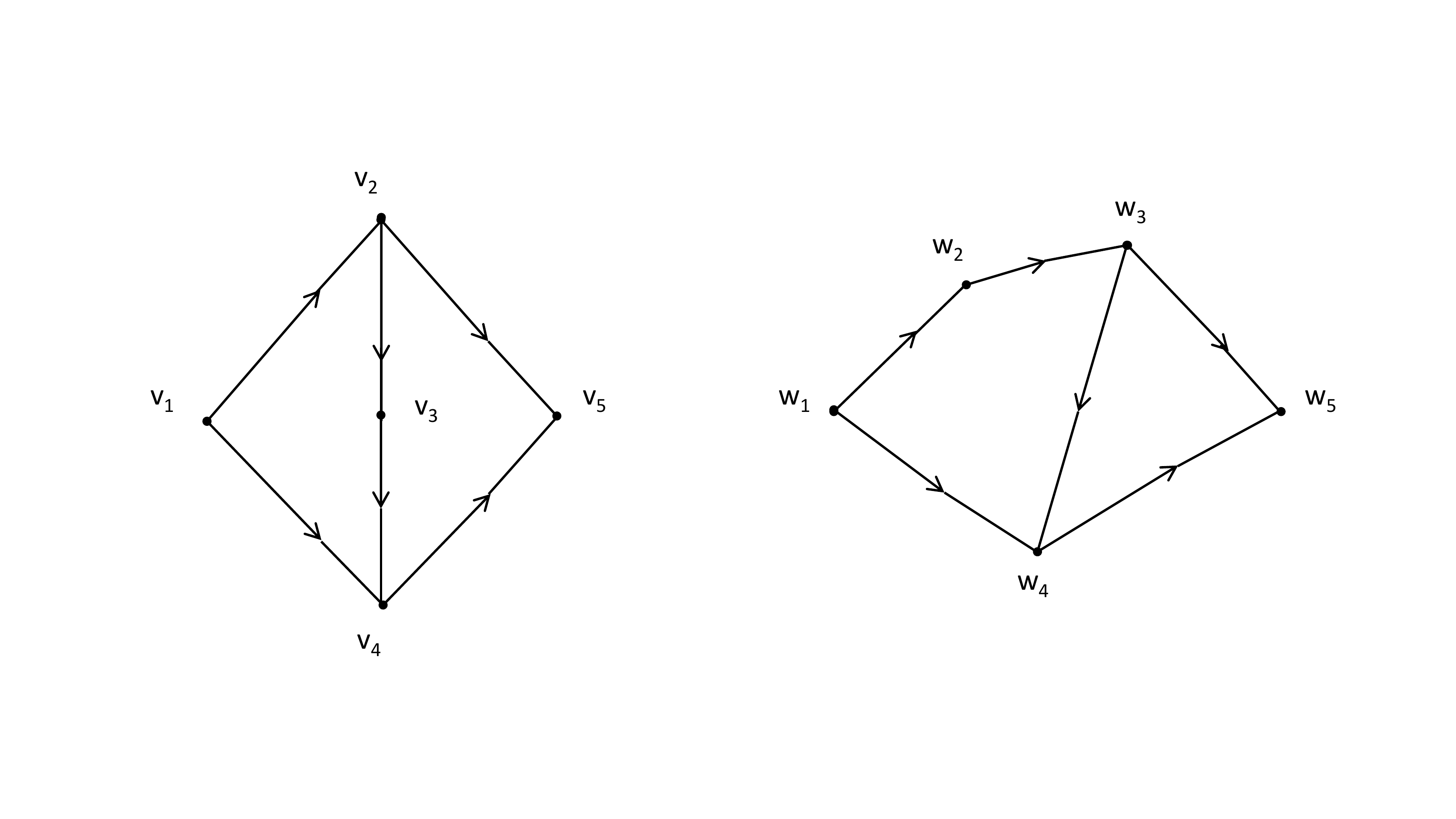}
\caption{Being reduced does not imply being strongly reduced and being strongly reduced does not imply being extremely reduced.}
\label{Ejp_1}
\end{figure}

\begin{definicion} Given a finite DAG $\vec{G}$ and a vertex $v\in \mathcal{V}$ we say that $w$ is an \textbf{ancestor} of $v$ if there is a directed path $\{w=v_0,...,v_k=v\}$ and $w$ is a \textbf{descendant} of $v$ if there is a directed path $\{v=v_0,...,v_k=w\}$. 
\end{definicion}

\begin{definicion} We say that a finite DAG $\vec{G}$ is \emph{extremely reduced} if for every pair of non-adjacent vertices $x,y$, if $x,y$ have a common ancestor, then they do not have a common descendant.
\end{definicion}

\begin{prop}\label{Prop: light-reduced} If a DAG $\vec{G}=(\mathcal{V},\vec{\mathcal{E}})$ is extremely reduced, then it is strongly reduced.
\end{prop}

\begin{proof} Let $\gamma=\{v,v_1,...,v_n,w\}$ and $\gamma'=\{v,w_0,...,w_m,w\}$ two directed paths in $\vec{G}$ from $v$ yo $w$. Let $\mathcal{O}$ be any topological order in $\vec{G}$ and consider $\gamma \cup_\mathcal{O} \gamma'=\{v,z_1,..,z_k,w\}$. First, notice that $z_1$ is either $v_1$ or $w_1$. Therefore, $\{v,z_1\}\in \vec{\mathcal{E}}$. Also, $z_k$ is either $v_n$ or $w_m$, and $\{z_k,w\}\in \vec{\mathcal{E}}$. Now, for every $1< i \leq k$, let us see that $\{z_{i-1},z_i\}\in \vec{\mathcal{E}}$.
If $z_{i-1},z_i\in \gamma$ or $z_{i-1},z_i\in \gamma'$, then they are consecutive vertices in a directed path and we are done. Otherwise, since $z_{i-1},z_i$ have a common ancestor $v$ and a common descendant $w$, then there is a directed edge joining them and, since 
$z_{i-1},z_i$ are sorted by a topological order,  $\{z_{i-1},z_i\}\in \vec{\mathcal{E}}$.
\end{proof}

\begin{obs} The converse is not true. The graph in the right from Figure \ref{Ejp_1} b), is strongly reduced. However, vertices $w_2$ and $w_4$ are not adjacent and have a common ancestor and a common descendent.
\end{obs}

\begin{prop}\label{Prop: equiv} If $\vec{G}$ is transitive, then the following properties are equivalent:
\begin{itemize}
	\item $\vec{G}$ is extremely reduced,
	\item $\vec{G}$ is strongly reduced,
	\item $\vec{G}$ is reduced.
\end{itemize}
\end{prop}

\begin{proof} By proposition \ref{Prop: light-reduced} if  $\vec{G}$ is extremely reduced, then it is strongly reduced. By Proposition \ref{Prop: strongly-reduced}, if  $\vec{G}$ is strongly reduced, then it is reduced.

Suppose  $\vec{G}$ is reduced and suppose that two vertices $x,y$ have a common ancestor, $v$, and a common descendant, $w$. Then, there are two directed paths $\gamma,\gamma'$ from $v$ to $w$ such that $x\in \gamma$ and $y\in \gamma'$. By property $i)$ in \ref{Prop: reduced}, there exist a path 
$\gamma''$ in $\vec{G}$ from $v$ to $w$ such that $\gamma,\gamma'\leq \gamma''$. In particular, $x,y\in \gamma''$. Therefore, either $x$ is reachable from $y$ or $y$ is reachable from $x$ in $\vec{G}$. Since $\vec{G}$ is transitive, this implies that $x,y$ are adjacent. Therefore, $\vec{G}$ is extremely reduced. 
\end{proof}

\begin{definicion} Given a DAG $\vec{G}=(\mathcal{V},\vec{\mathcal{E}})$, the graph with vertex set $\mathcal{V}$ and edge set $\vec{\mathcal{E}'}:=\vec{\mathcal{E}}\cup \{\{v,w\} \, | \ w \mbox{ is reachable from } v\}$ is called the \emph{transitive closure} of $\vec{G}$, 
$T[\vec{G}]$.
\end{definicion}

It is immediate to check the following:

\begin{prop}\label{Prop: transitive} Given any DAG $\vec{G}$, $T[\vec{G}]$ is transitive.
\end{prop}

\begin{prop}\label{Prop: closure} If a DAG $\vec{G}$ is reduced, then the transitive closure $T[\vec{G}]$ is also reduced.
\end{prop}

\begin{proof} Suppose $\vec{G}$ satisfies $i)$ and let $\gamma=\{v=v_0,...,v_n=w\}$, $\gamma'=\{v=w_0,...,w_m=w\}$ be any pair of paths in $T[\vec{G}]$. Therefore, $v_i$ is reachable from $v_{i-1}$ in $\vec{G}$ for every $1\leq i \leq n$ and $w_i$ is reachable from $w_{i-1}$ in $\vec{G}$ for every $1\leq i \leq m$. Thus, there exist a sequence $\gamma_0$ in $\vec{G}$ such that $\gamma\leq \gamma_0$ and a sequence $\gamma_0'$ in $\vec{G}$ such that $\gamma'\leq \gamma_0'$. By property $i)$, there is a directed path from $v$ to $w$ such that $\gamma_0,\gamma_0'\leq \gamma_0''$. Therefore, $\gamma,\gamma'\leq \gamma_0''$ and $T[\vec{G}]$ satisfies $i)$.
\end{proof}

Then, from propositions  \ref{Prop: strongly-reduced}, \ref{Prop: light-reduced}, \ref{Prop: equiv} and \ref{Prop: closure},

\begin{cor}\label{Cor: closure} If a DAG $\vec{G}$ is reduced, then the transitive closure $T[\vec{G}]$ is extremely reduced and strongly reduced. In particular, if $\vec{G}$ is extremely reduced or strongly reduced, then $T[\vec{G}]$ is extremely reduced and strongly reduced.
\end{cor}

Let us recall that 

\begin{equation}
T(n,\ell,1)=\binom{n}{2}-\binom{n-\ell+1}{2}=(n-\ell+1)(\ell-1)+\frac{(\ell-1)(\ell-2)}{2}
\end{equation}

As it was proved in \cite{MMO}, 

\begin{lema} \label{restaT} For $n\geq \ell$ and $d \geq 1$,
$$T(n+d,\ell,1) - T(n,\ell,1) = d(\ell-1).$$
\end{lema}

In particular, $T(n+2,\ell,1) - T(n,\ell,1) = 2(\ell-1)$.

Also, from \cite{MMO},

\begin{lema}\label{restat} For $1 \leq d \leq n$, 
$$t(n+d,d) - t(n,d) = (d-1)n +  \binom{d}{2}$$
\end{lema}

In particular, $t(n+2,2) - t(n,2) = n +  1$.

\begin{teorema}\label{Theorem: bound} Let $\vec{G}=(\mathcal{V},\vec{\mathcal{E}})$ be DAG with $n$ vertices and such that the longest directed path has length $\ell\geq 1$. If $\vec{G}$ is extremely reduced, then $\vec{G}$ has at most $t(n-\ell+1,2)+T(n,\ell,1)$ edges. 
\end{teorema}

\begin{proof} Let us prove the result by induction on $n$. Suppose that the longest directed path has length $\ell$. 

First, let us see that the result is true for $n=\ell+1$ and $n=\ell+2$.

If $n=\ell+1$ and there is a directed path of length $\ell$ then $\vec{G}$ has at most $\frac{\ell(\ell+1)}{2}=\frac{(\ell-2)(\ell-1)}{2}+2(\ell-1)+1=T(n,\ell,1)+t(n-\ell+1,2)$ edges. 

If $n=\ell+2$ and there is a directed path of length $\ell$ then there are $\ell+1$ vertices which define a directed path $\gamma=\{v_0,...,v_\ell\}$ and one vertex $w$ such that neither $\{w,v_0\}$ nor $\{v_\ell,w\}$ is a directed edge. Then, the partition $P_\Gamma=\{V_0,...,V_\ell\}$ of $\vec{G}$ satisfies that $v_i\in V_i$ for every $0\leq i \leq \ell$. Also, $w\in V_j$ for some $0\leq j \leq \ell$ and $\{w,v_j\}$, $\{v_j,w\}$ are not directed edges. Hence, $deg(w)\leq \ell$. 
Therefore, $\vec{G}$ has at most $\frac{\ell(\ell+1)}{2}+\ell=\frac{(\ell-2)(\ell-1)}{2}+3(\ell-1)+2=T(n,\ell,1)+t(n-\ell+1,2)$ edges. 

Suppose the induction hypothesis holds when the graph has $n$ vertices and let $\#(\mathcal{V})=n+2$. Also, by Proposition \ref{Prop: closure} we may assume that the graph is transitive. 

Consider the partition $P_\Gamma=\{V_0,...,V_\ell\}$ of $\mathcal{V}$. Let $\#(V_i)=r_i$. Let $v\in V_0$ and $w$ be any sink of $\vec{G}$. 
Consider any pair of vertices $v_i,v'_i\in V_i$. Since $\vec{G}$ is extremely reduced and every two vertices in $V_i$ are non-adjacent, $v_i,v'_i$ can not be both descendants from $v$ and ancestors for $w$ simultaneously. Hence, the number of edges joining the sets $\{v,w\}$ and $V_i$ are at most $r_i+1$. Therefore, there are at most $n+\ell-1$ edges joining $\{v,w\}$ and $G\backslash \{v,w\}$

Since $G\backslash \{v,w\}$ has $n$ vertices, by hypothesis, it contains at most $t(n-\ell+1,2)+T(n,\ell,1)$ edges.

Finally, there is at most 1 edge in the subgraph induced by $\{v,w\}$.

Therefore, by lemmas \ref{restaT} and \ref{restat}, $\#(\vec{E}(G))\leq t(n-\ell+1,2)+T(n,\ell,1)+n+\ell=t(n-\ell+3,2)+T(n+2,\ell,1)$.
\end{proof}

By Corollary \ref{Cor: closure} we know that the extremal graph for reduced and strongly reduced graphs is transitive. Thus, from Theorem \ref{Theorem: bound} and Proposition \ref{Prop: equiv} we obtain the following. 

\begin{cor} Let $\vec{G}=(\mathcal{V},\vec{\mathcal{E}})$ be DAG with $n$ vertices and such that the longest directed path has length $\ell\geq 1$. If $\vec{G}$ is reduced, then $\vec{G}$ has at most $t(n-\ell+1,2)+T(n,\ell,1)$ edges. 
\end{cor}

\begin{cor} Let $\vec{G}=(\mathcal{V},\vec{\mathcal{E}})$ be DAG with $n$ vertices and such that the longest directed path has length $\ell\geq 1$. If $\vec{G}$ is strongly reduced, then $\vec{G}$ has at most $t(n-\ell+1,2)+T(n,\ell,1)$ edges. 
\end{cor}

\section{Directed intersection graphs of boxes}

\begin{definicion} Let $\mathcal{R}$ be a collection of boxes with parallel axis in $\br^2$. Let $\vec{G}=(\mathcal{V},\vec{\mathcal{E}})$ be a directed graph such that $\mathcal{V}=\mathcal{R}$ and given $R,R'\in \mathcal{R}$ with $R=I\times J$, $R'=I'\times J'$ then 
$\{R,R'\}\in \vec{\mathcal{E}}$ if and only if $I\subset I'$ and $J'\subset J$ (i.e. there is an edge if and only if the intersection is \textit{transverse} and the order is defined by the subset relation in the first coordinate). Let us call $\vec{G}$ the \textbf{directed intersection graph} of $\mathcal{R}$.
\end{definicion}

\begin{figure}[ht]
\centering
\includegraphics[scale=0.3]{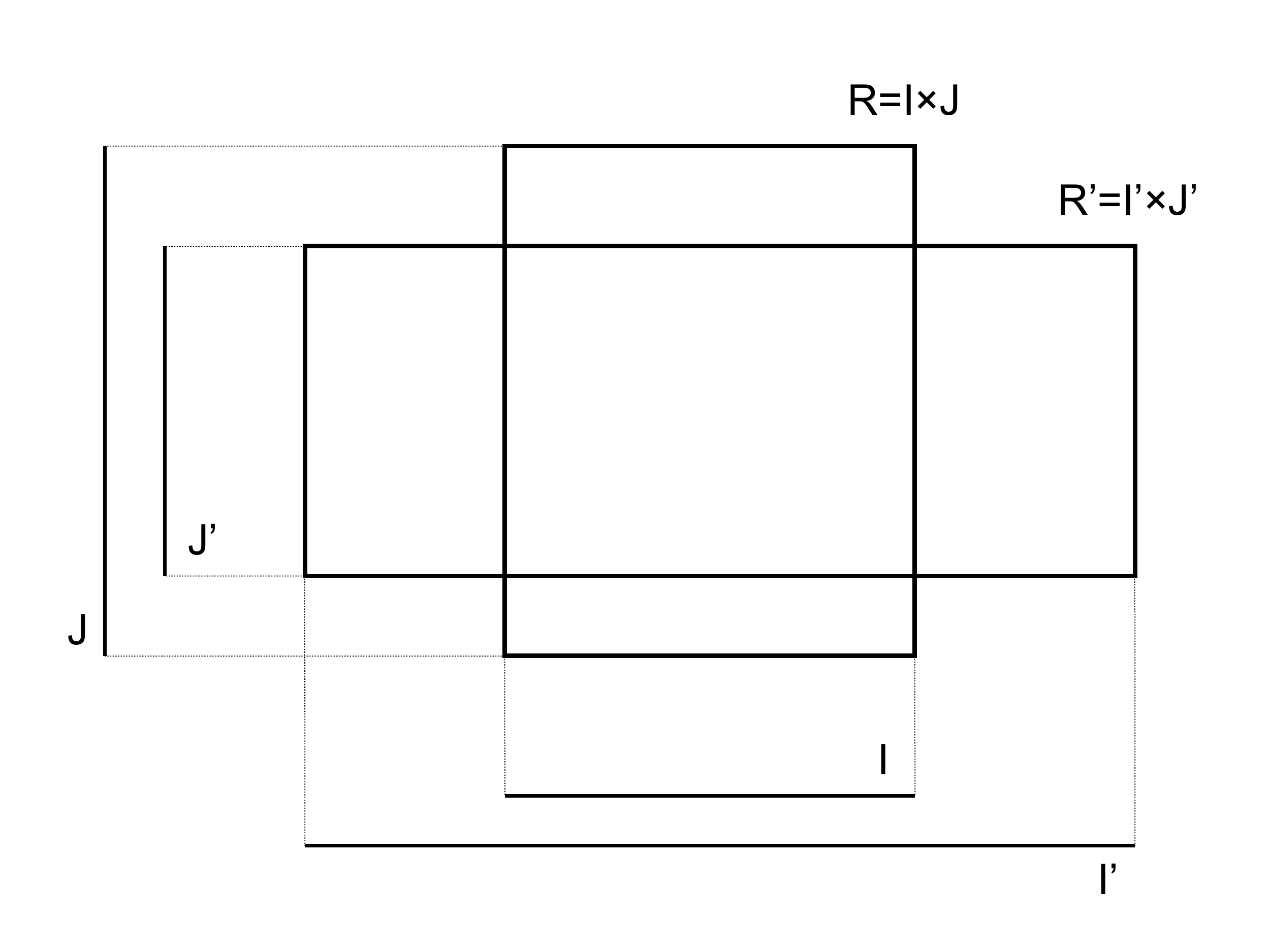}
\caption{The transverse intersection above induces a directed edge $\{R,R'\}$.}
\end{figure}

\begin{definicion} Let $\mathcal{R}$ be a collection of boxes with parallel axis in $\br^2$. We say that $\mathcal{R}$ is a collection  with \textbf{transverse intersection} if for every pair of boxes either they are disjoint or their intersection is transverse.
\end{definicion}

\begin{prop}\label{Prop: intersect} Let $\mathcal{R}$ be a collection of boxes with parallel axis in $\br^2$  and $\vec{G}$ be the induced directed intersection graph. If two vertices $v$, $w$ have both a common ancestor and a common descendant in $\vec{G}$, then the corresponding boxes $R_v,R_w$ intersect. 
\end{prop}

\begin{proof} Let $a$ be a common ancestor and $R_a=I_a\times J_a$ be the corresponding box. Let $b$ be a common descendant and $R_b=I_b\times J_b$ be the corresponding box. Then if $R_v=I_v\times J_v$, $R_w=I_w\times J_w$ are the boxes corresponding to $v$ and $w$ respectively, it follows by construction that $I_a\subset I_v,I_w$ and $J_b\subset J_v,J_w$. Therefore, $I_a\times J_b\subset R_v,R_w$ and $R_v\cap R_w\neq \emptyset$.
\end{proof}

\begin{prop}\label{Prop: common} If $\mathcal{R}$ is a collection of boxes with parallel axis in $\br^2$ with transverse intersection, then the induced directed intersection graph $G$ is extremely reduced and transitive. 
\end{prop}

\begin{proof} Let $v,w$ be two vertices such that there is no edge joining them. This means, by construction, that their corresponding boxes do not have a transverse intersection. Since $\mathcal{R}$ has transverse intersection, this implies that  these boxes do not intersect. Thus, by Proposition \ref{Prop: intersect}, if $v,w$ have a common ancestor, then they can not have a common descendant.
\end{proof}

\begin{obs}
Consider the bipartite graph $G$ from Figure \ref{contraexample} with the partition given by  \{letters, numbers\} and assume all directed edges go from letters into numbers. Note that $G$ is extremely reduced, transitive and acyclic.
It is not difficult to observe that the induced subgraph given by the set of vertices $\{1,2,3,4,8,9, A,B,C,D,H,I\}$ is realizable as boxes in $\br^2$, or what is equivalent in this case, by intervals in the plane conforming two sets of disjoint squares, one given by $A,B,1,2$ and the other by $3,4,C,D$, one strictly inside the other.  Then by the same observation applied to the induced subgraphs given by the set of vertices $\{1,2,5,6,A,B,E,F,7,12,G,L\}$ and the set of vertices $\{3,4,5,6,C,D,E,F,10,11,J,K\}$  it is forced necessarily a system of tree squares one inside the other. However, intervals given by $\{ 7,8,9,10,11,12\} $ and $\{ G,H,I,J,K,L\}$ are forced to have more intersections that those given by the graph.
In other words, there is no family of boxes (or intervals) that realizes such a graph or for which it is induced the graph $G$.  Then, the converse of Proposition \ref{Prop: common} is not true.  
\end{obs}

\begin{figure}
\centering
\includegraphics[scale=0.3]{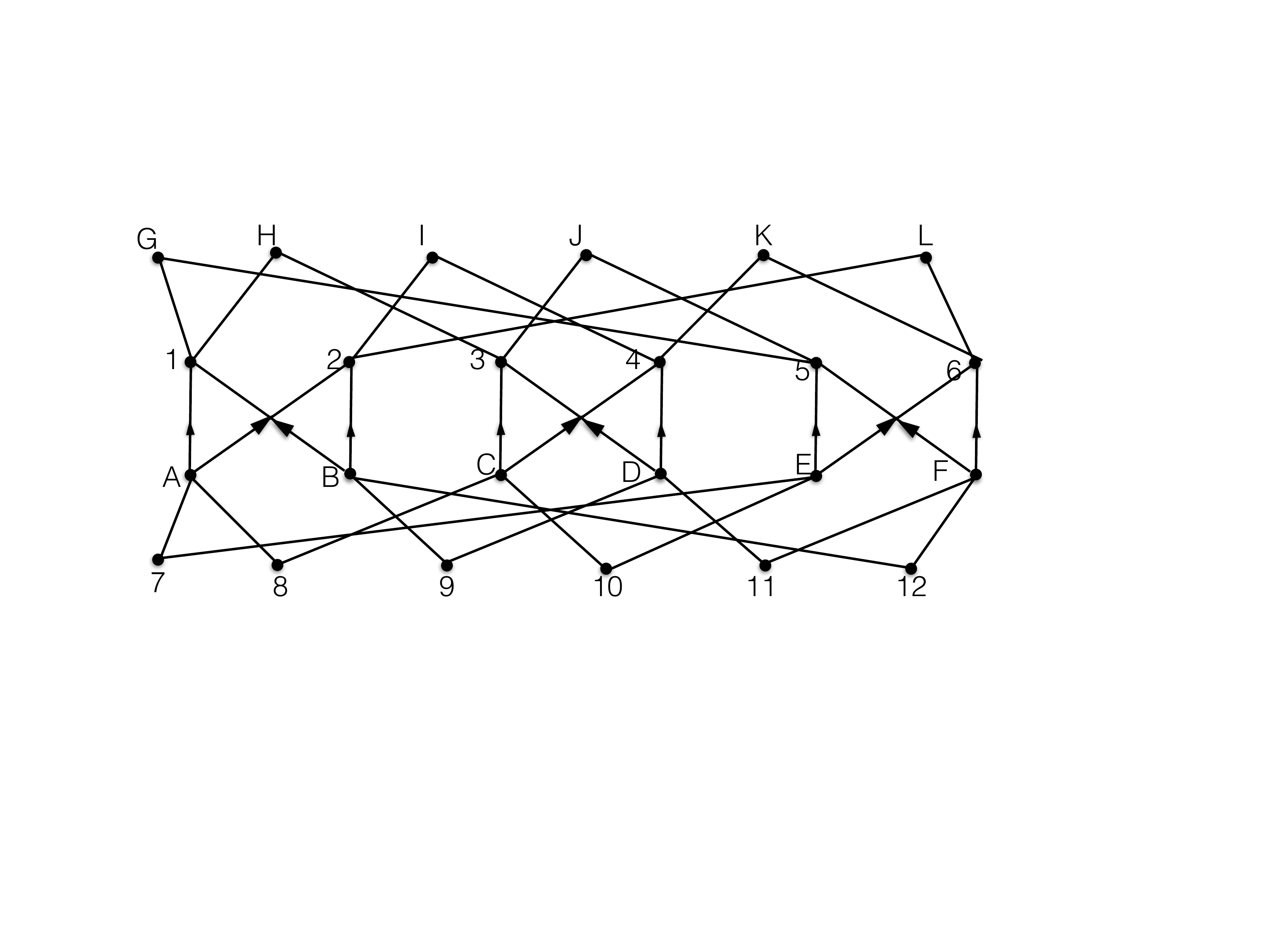}
\caption{The bipartite transitive and extremely reduced DAG, $G$ with partition given by  \{letters, numbers\} and edges directed from letters into numbers. That is not realizable as a family of boxes in $\br^2$} 
\label{contraexample}
\end{figure}

\begin{figure}[ht]
\centering
\includegraphics[scale=0.4]{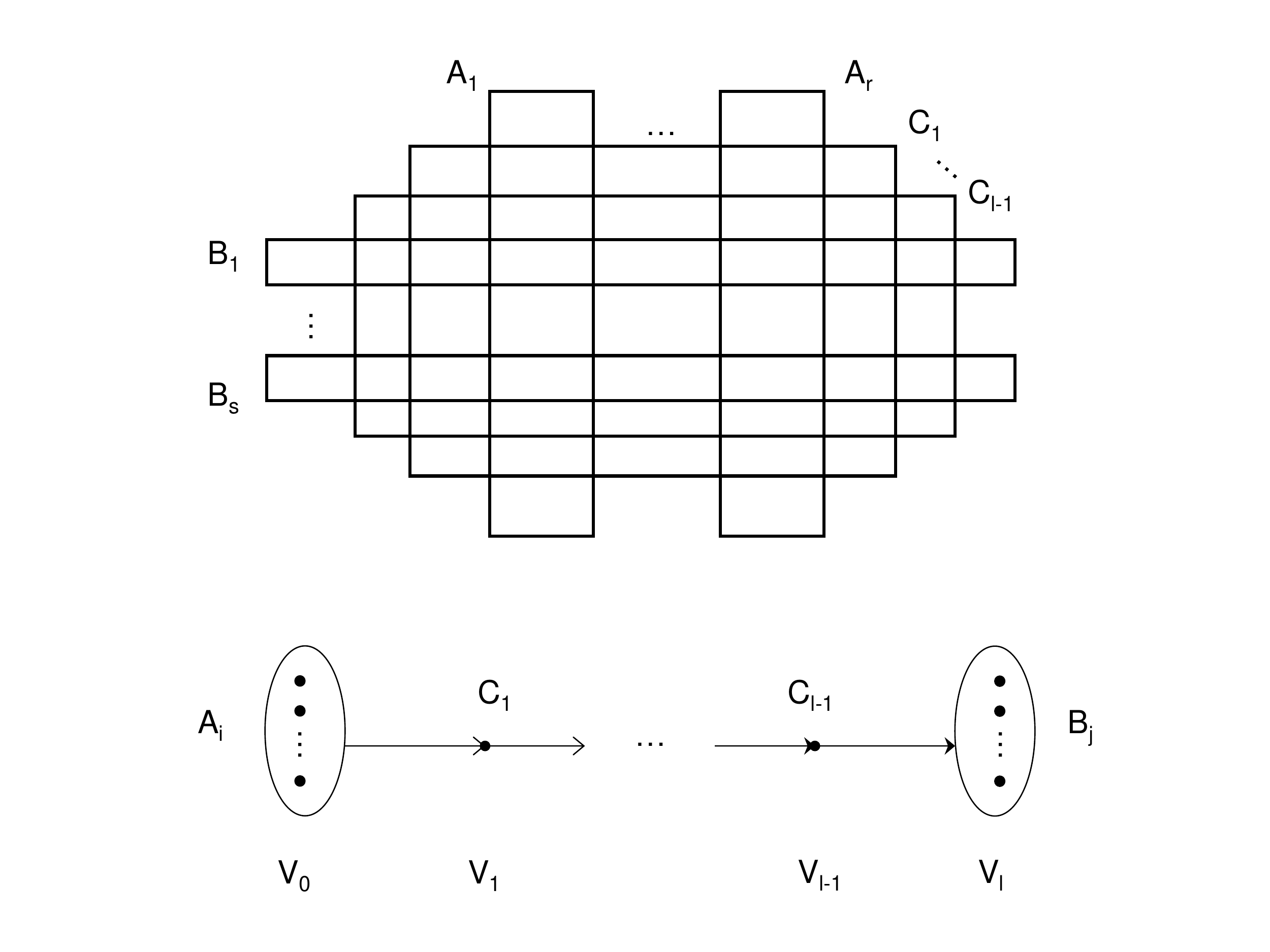}
\caption{The graph $G[r,l,s]$ corresponds to the directed intersection graph of the collection in the figure where $x_i\sim A_i$, $y_j\sim C_j$ and $z_k\sim B_k$. Notice that the graph is transitive although not every edge is represented in the figure.}
\label{Figure: extremal}
\end{figure}

Let $G[r,l,s]$ be the graph, $G(\mathcal{V},\vec{\mathcal{E}})$, such that: 

$\mathcal{V}=\{x_1,...,x_r,y_1,...,y_{l-1},z_1,...,z_s\}$

$\{x_i,x_j\}\notin \vec{\mathcal{E}}$ for any $i\neq j$,

$\{z_i,z_j\}\notin \vec{\mathcal{E}}$ for any $i\neq j$,

$\{x_i,y_j\}\in \vec{\mathcal{E}}$ for every $i, j$,

$\{y_i,y_j\}\in \vec{\mathcal{E}}$ for every $i< j$,

$\{y_i,z_j\}\in \vec{\mathcal{E}}$ for every $i, j$,

$\{x_i,z_j\}\in \vec{\mathcal{E}}$ for every $i, j$.

This is the directed intersection graph from the collection of boxes in Figure \ref{Figure: extremal}.

By Proposition \ref{Prop: common}, $G[r,l,s]$ is a transitive extremely reduced DAG. In particular, $G[r,l,s]$ is strongly reduced and reduced.

Now, to prove that the bound obtained in Theorem \ref{Theorem: bound} and its corollaries is best possible, it is immediate to check the following:

\begin{prop} If $n-\ell$ is even, $G[\frac{n-\ell}{2},\ell,\frac{n-\ell}{2}]$ has $t(n-\ell+1,2)+T(n,\ell,1)$ edges. If $n-\ell$ is odd, 
$G[\frac{n-\ell+1}{2},\ell,\frac{n-\ell-1}{2}]$ has $t(n-\ell+1,2)+T(n,\ell,1)$ edges.
\end{prop}

\section{Acknowledgements}
The first author was partially supported by MTM 2012-30719. The second and third authors wish to acknowledge support from PAPIIT IN105915 and
CONACyT 166306.

\end{document}